\documentclass[11pt]{amsart}
\usepackage{amsmath}
\usepackage[active]{srcltx}
\usepackage{t1enc}
\usepackage[latin2]{inputenc}
\usepackage{verbatim}
\usepackage{amsmath,amsfonts,amssymb,amsthm}
\usepackage[mathcal]{eucal}
\usepackage{enumerate}
\usepackage[centertags]{amsmath}
\usepackage{graphics}

\newtheorem{theorem}{Theorem}

\newtheorem*{OldTheorem}{Theorem A}
\newtheorem*{OldTheorem2}{Theorem B}
\newtheorem*{OldTheorem3}{Theorem C}
\newtheorem*{OldTheorem4}{Theorem D}
\newtheorem*{OldTheorem5}{Theorem E}

\begin{document}
\author{Gy\"orgy G\'at and Ushangi Goginava}
\title[ Almost Everywhere Strong Summability]{ Almost Everywhere Strong
Summability of Double Walsh-Fourier Series}
\address{G. G\'at, Institute of Mathematics and Computer Science, College of
Ny\'\i regyh\'aza, P.O. Box 166, Nyiregyh\'aza, H-4400 Hungary }
\email{gatgy@nyf.hu}
\address{U. Goginava, Department of Mathematics, Faculty of Exact and
Natural Sciences, Ivane Javakhishvili Tbilisi State University, Chavcha\-vadze str. 1, Tbilisi
0128, Georgia}
\email{zazagoginava@gmail.com}
\date{}
\maketitle

\begin{abstract}
In this paper we study the a. e. strong convergence of the quadratical
partial sums of the two-dimensional Walsh-Fourier series. Namely, we prove the a.e. relation
$(\frac{1}{n}\sum\limits_{m=0}^{n-1}\left\vert S_{mm}f - f \right\vert^{p})^{1/p}\rightarrow 0$
for every  two-dimensional functions belonging to $L\log L$ and $0<p\le 2$. From the theorem of Getsadze \cite{Gets} it follows that the space $L\log L$ can not be enlarged with preserving this strong summability property.
\end{abstract}

\medskip

\footnotetext{%
\par
Research was supported by project T\'AMOP-4.2.2.A-11/1/KONV-2012-0051 and by
Shota Rustaveli National Science Foundation grant no.31/48 (Operators in some function spaces and their applications in
Fourier analysis)}

\section{ Introduction}

Let $\mathbb{P}$ denote the set of positive integers, $\mathbb{N}\mathbf{:=}%
\mathbb{P}\mathbf{\cup \{}0\mathbf{\}.}$ Denote $\mathbb{Z}_{2}$ the
discrete cyclic group of order 2, that is $\mathbb{Z}_{2}=\{0,1\},$ where
the group operation is the modulo 2 addition and every subset is open. The
Haar measure on $\mathbb{Z}_{2}$ is given such that the measure of a
singleton is 1/2. Let $G$ be the complete direct product of the countable
infinite copies of the compact groups $\mathbb{Z}_{2}.$ The elements of $G$
are of the form $x=\left( x_{0},x_{1},...,x_{k},...\right) $ with $x_{k}\in
\{0,1\}\left( k\in \mathbb{N}\right) .$ The group operation on $G$ is the
coordinate-wise addition, the measure (denote\thinspace $\,$by$\,\,\mu $)
and the topology are the product measure and topology. The compact Abelian
group $G$ is called the Walsh group. A base for the neighborhoods of $G$ can
be given in the following way:
\begin{eqnarray*}
I_{0}\left( x\right) &:&=G,\,\,\,I_{n}\left( x\right) :=\,I_{n}\left(
x_{0},...,x_{n-1}\right) \\
&:&=\left\{ y\in G:\,y=\left( x_{0},...,x_{n-1},y_{n},y_{n+1},...\right)
\right\} ,
\end{eqnarray*}%
\begin{equation*}
\,\left( x\in G,n\in \mathbb{N}\right) .
\end{equation*}%
These sets are called the dyadic intervals. Let $0=\left( 0:i\in \mathbb{N}%
\right) \in G$ denote the null element of $G,\,\,\,I_{n}:=I_{n}\left(
0\right) \,\left( n\in \mathbb{N}\right) .$ Set $e_{n}:=\left(
0,...,0,1,0,...\right) \in G$ the $n\,$th\thinspace coordinate of which is 1
and the rest are zeros $\left( n\in \mathbb{N}\right) .$

For $k\in \mathbb{N}$ and $x\in G$ denote by
\begin{equation*}
r_{k}\left( x\right) :=\left( -1\right) ^{x_{k}}\,\,\,\,\,\,\left( x\in
G,k\in \mathbb{N}\right)
\end{equation*}%
the $k$-th Rademacher function. If $n\in \mathbb{N}$, then $%
n=\sum\limits_{i=0}^{\infty }n_{i}2^{i},$ where $n_{i}\in \{0,1\}\,\,\left(
i\in \mathbf{N}\right) $, i. e. $n$ is expressed in the number system of
base $2$. For $n>0$ denote $\left\vert n\right\vert :=\max \{j\in \mathbb{N} \mathbf{:} %
n_{j}\neq 0\}$, that is, $2^{\left\vert n\right\vert }\leq n<2^{\left\vert
n\right\vert +1}.$

The Walsh-Paley system is defined as the sequence of Walsh-Paley functions:
\begin{equation*}
w_{n}\left( x\right) :=\prod\limits_{k=0}^{\infty }\left( r_{k}\left(
x\right) \right) ^{n_{k}}= \left(
-1\right) ^{\sum\limits_{k=0}^{\left\vert n\right\vert}n_{k}x_{k}}\,\,\,\,\,\,\left( x\in G,n\in \mathbb{P}\right),
\end{equation*}
and $w_0:=1$. The Walsh-Dirichlet kernel is defined by
\begin{equation*}
D_{n}\left( x\right) =\sum\limits_{k=0}^{n-1}w_{k}\left( x\right) .
\end{equation*}

Recall that (see \cite{GES} and \cite{SWS})
\begin{equation}
D_{2^{n}}\left( x\right) =\left\{
\begin{array}{c}
2^{n},\mbox{if }x\in I_{n}, \\
0,\,\,\,\mbox{if }x\in \overline{I}_{n}.%
\end{array}%
\right.
\end{equation}

We consider the double system $\left\{ w_{n}(x)\times w_{m}(y):\,n,m\in
\mathbf{N}\right\} $ on the $G\times G.$ The notiation $a\lesssim b$ in the
whole paper stands for $a\leq c\cdot b$, where $c$ is an absolute constant.

The rectangular partial sums of the 2-dimensional Walsh-Fourier series are
defined as follows:

\begin{equation*}
S_{M,N}(x,y,f):=\sum\limits_{i=0}^{M-1}\sum\limits_{j=0}^{N-1}\widehat{f}%
\left( i,j\right) w_{i}\left( x\right) w_{j}\left( y\right) ,
\end{equation*}%
where the number
\begin{equation*}
\widehat{f}\left( i,j\right) =\int\limits_{G\times G}f\left( x,y\right)
w_{i}\left( x\right) w_{j}\left( y\right) d\mu \left( x,y\right)
\end{equation*}%
is said to be the $\left( i,j\right) $th Walsh-Fourier coefficient of the
function\thinspace $f.$

Denote
\begin{equation*}
S_{n}^{\left( 1\right) }\left( x,y,f\right) :=\sum\limits_{l=0}^{n-1}%
\widehat{f}\left( l,y\right) w_{l}\left( x\right) ,
\end{equation*}%
\begin{equation*}
S_{m}^{\left( 2\right) }\left( x,y,f\right) :=\sum\limits_{r=0}^{m-1}%
\widehat{f}\left( x,r\right) w_{r}\left( y\right) ,
\end{equation*}%
where
\begin{equation*}
\widehat{f}\left( l,y\right) =\int\limits_{G}f\left( x,y\right) w_{l}\left(
x\right) d\mu \left( x\right)
\end{equation*}%
and
\begin{equation*}
\widehat{f}\left( x,r\right) =\int\limits_{G}f\left( x,y\right) w_{r}\left(
y\right) d\mu \left( y\right) .
\end{equation*}

The norm (or quasinorm) of the space $L_{p}\left( G\times G\right) $ is
defined by
\begin{equation*}
\left\Vert f\right\Vert _{p}:=\left( \int\limits_{G\times G}\left\vert
f\left( x,y\right) \right\vert ^{p}d\mu \left( x,y\right) \right)
^{1/p}\,\,\,\,\left( 0<p<+\infty \right) .
\end{equation*}

We denote by $L\log L\left( G\times G\right) $ the class of measurable
functions $f$, with
\begin{equation*}
\int\limits_{G\times G}|f|\log ^{+}|f|<\infty ,
\end{equation*}%
where $\log ^{+}u:=\mathbb{I}_{(1,\infty )}\log u,$ where $\mathbb{I}_{E}$
is character function of the set $E$.

Denote by $S_{n}^{T}(x,f)$ the partial sums of the trigonometric Fourier
series of $f$ and let
\begin{equation*}
\sigma _{n}^{T}(x,f)=\frac{1}{n+1}\sum_{k=0}^{n}S_{k}^{T}(x,f)
\end{equation*}%
be the $(C,1)$ means. Fejér \cite{Fe} proved that $\sigma _{n}^{T}(f)$
converges to $f$ uniformly for any $2\pi $-periodic continuous function.
Lebesgue in \cite{Le} established almost everywhere convergence of $(C,1)$
means if $f\in L_{1}(\mathbb{T}),\mathbb{T}:=[-\pi ,\pi )$. The strong
summability problem, i.e. the convergence of the strong means
\begin{equation}
\frac{1}{n+1}\sum\limits_{k=0}^{n}\left\vert S_{k}^{T}\left( x,f\right)
-f\left( x\right) \right\vert ^{p},\quad x\in \mathbb{T},\quad p>0,
\label{Hp}
\end{equation}%
was first considered by Hardy and Littlewood in \cite{H-L}. They showed that
for any $f\in L_{r}(\mathbb{T})~\left( 1<r<\infty \right) $ the strong means
tend to $0$ a.e., if $n\rightarrow \infty $. The Fourier series of $f\in
L_{1}(\mathbb{T})$ is said to be $\left( H,p\right) $-summable at $x\in T$,
if the values (\ref{Hp}) converge to $0$ as $n\rightarrow \infty $. The $%
\left( H,p\right) $-summability problem in $L_{1}(\mathbb{T})$ has been
investigated by Marcinkiewicz \cite{Ma} for $p=2$, and later by Zygmund \cite%
{Zy2} for the general case $1\leq p<\infty $. Oskolkov in \cite{Os} proved
the following: Let $f\in L_{1}(\mathbb{T})$ and let $\Phi $ be a continuous
positive convex function on $[0,+\infty )$ with $\Phi \left( 0\right) =0$
and
\begin{equation}
\ln \Phi \left( t\right) =O\left( t/\ln \ln t\right) \text{ \ \ \ }\left(
t\rightarrow \infty \right) .  \label{a1}
\end{equation}%
Then for almost all $x$%
\begin{equation}
\lim\limits_{n\rightarrow \infty }\frac{1}{n+1}\sum\limits_{k=0}^{n}\Phi
\left( \left\vert S_{k}^{T}\left( x,f\right) -f\left( x\right) \right\vert
\right) =0.  \label{osk}
\end{equation}

It was noted in \cite{Os} that Totik announced the conjecture that (\ref{osk}%
) holds almost everywhere for any $f\in L_{1}(\mathbb{T})$, provided
\begin{equation}
\ln \Phi \left( t\right) =O\left( t\right) \quad \left( t\rightarrow \infty
\right) .  \label{a2}
\end{equation}%
In \cite{Ro} Rodin proved

\begin{OldTheorem}
Let $f\in L_{1}(\mathbb{T})$. Then for any $A>0$
\begin{equation*}
\lim\limits_{n\rightarrow \infty }\frac{1}{n+1}\sum\limits_{k=0}^{n}\left(
\exp \left( A\left\vert S_{k}^{T}\left( x,f\right) -f\left( x\right)
\right\vert \right) -1\right) =0
\end{equation*}%
for a. e. $x\in \mathbb{T}$.
\end{OldTheorem}

Karagulyan \cite{Ka} proved that the following is true.

\begin{OldTheorem2}
Suppose that a continuous increasing function $\Phi :[0,\infty )\rightarrow
\lbrack 0,\infty ),\Phi \left( 0\right) =0$, satisfies the condition%
\begin{equation*}
\limsup_{t\rightarrow +\infty }\frac{\log \Phi \left( t\right) }{t}=\infty .
\end{equation*}%
Then there exists a function $f\in L_{1}(\mathbb{T})$ for which the relation%
\begin{equation*}
\limsup_{n\rightarrow \infty }\frac{1}{n+1}\sum\limits_{k=0}^{n}\Phi \left(
\left\vert S_{k}^{T}\left( x,f\right) \right\vert \right) =\infty
\end{equation*}%
holds everywhere on $\mathbb{T}$.
\end{OldTheorem2}

For quadratic partial sums of two-dimensional trigonometric Fourier series
Marcinkiewicz \cite{Ma2} has proved, that if $f\in L\log L\left( \mathbb{T}%
^{2}\right) $,$\mathbb{T}:=[-\pi ,\pi )^{2}$, then%
\begin{equation*}
\lim\limits_{n\rightarrow \infty }\frac{1}{n+1}\sum\limits_{k=0}^{n}\left(
S_{kk}^{T}\left( x,y,f\right) -f\left( x,y\right) \right) =0
\end{equation*}%
for a. e. $\left( x,y\right) \in \mathbb{T}^{2}$. ~Zhizhiashvili \cite{Zh}
improved this result showing that class $L\log L\left( \mathbb{T}^{2}\right)
$ can be replaced by $L_{1}\left( \mathbb{T}^{2}\right) $.

From a result of ~Konyagin \cite{Kon} it follows that for every $\varepsilon
>0$ there exists a function $f\in L\log ^{1-\varepsilon }\left( \mathbb{T}%
^{2}\right) $ such that
\begin{equation}
\lim\limits_{n\rightarrow \infty }\frac{1}{n+1}\sum\limits_{k=0}^{n}\left%
\vert S_{kk}^{T}\left( x,y,f\right) -f\left( x,y\right) \right\vert \neq 0%
\text{ \ \ for a. e. }\left( x,y\right) \in \mathbb{T}^{2}.  \label{str}
\end{equation}

These results show that in the case of one dimensional functions the $(C,1)$ summability
and $(C,1)$ strong summability we have the same maximal convergence spaces. That is, in both cases we have $L_1$.
But, the situation changes as we step further to the case of two dimensional functions. In other words,
the spaces of functions with almost everywhere summable  Marcinkiewicz and  strong Marcinkiewicz means are different.

The results on strong summation and approximation of trigonometric Fourier
series have been extended for several other orthogonal systems. For
instance, concerning the Walsh system see Schipp \cite{Sch1,Sch2,Sch3},
Fridli \cite{FS,FS2}, Leindler \cite{Le,Le1,Le2,Le3,Le4}, Totik \cite%
{To1,To2,To3}, Fridli and Schipp \cite{FS2}, Rodin \cite{Ro1}, Weisz \cite%
{We,We2}, Gabisonia \cite{Ga}.

The problems of summability of cubical partial sums of multiple Fourier
series have been investigated by Gogoladze \cite{Gog1,Gog2,Gog3}, Wang \cite%
{Wa}, Zhag \cite{ZhHe}, Glukhov \cite{Gl}, Goginava \cite{Gogi}, G\'at,
Goginava, Tkebuchava \cite{GGT}, Goginava, Gogoladze \cite{GG} .

For Walsh system Schipp \cite{Sch} proved that the following is true.

\begin{OldTheorem3}
Let $f\in L_{1}(G)$. Then for any $A>0$
\begin{equation*}
\lim\limits_{n\rightarrow \infty }\frac{1}{n+1}\sum\limits_{k=0}^{n}\left(
\exp \left( A\left\vert S_{k}\left( x,f\right) -f\left( x\right) \right\vert
\right) -1\right) =0
\end{equation*}%
for a. e. $x\in G$.
\end{OldTheorem3}

Schipp in \cite{Sch} introduce the following operator%
\begin{equation*}
V_{n}f\left( x\right) :=\left( \frac{1}{2^{n}}\int\limits_{G}\left(
\sum\limits_{j=0}^{n-1}2^{j-1}\mathbb{I}_{I_{j}}\left( t\right)
S_{2^{n}}f\left( x+t+e_{j}\right) \right) ^{2}d\mu \left( t\right) \right)
^{1/2}.
\end{equation*}

Let%
\begin{equation*}
Vf:=\sup\limits_{n}V_{n}f.
\end{equation*}

The following theorem is proved by Schipp.

\begin{OldTheorem4}[\protect\cite{Sch}]
\label{SchD} Let $f\in L_{1}\left( G\right) $. Then%
\begin{equation*}
\mu \left\{ \left\vert Vf\right\vert >\lambda \right\} \lesssim \frac{%
\left\Vert f\right\Vert _{1}}{\lambda }.
\end{equation*}
\end{OldTheorem4}

In \cite{GG} it is studied the exponential uniform strong approximation of the
Marcinkiewicz means of the two-dimensional Walsh-Fourier series. We say that
the function $\psi $ belongs to the class $\Psi $ if it increase on $%
[0,+\infty )$ and
\begin{equation*}
\lim\limits_{u\rightarrow 0}\psi \left( u\right) =\psi \left( 0\right) =0.
\end{equation*}

\begin{OldTheorem5}[\protect\cite{GG}]
a)Let $\varphi \in \Psi $ and let the inequality
\begin{equation*}
\overline{\lim\limits_{u\rightarrow \infty }}\frac{\varphi \left( u\right) }{%
\sqrt{u}}<\infty
\end{equation*}%
hold. Then for any function $f\in C\left( G\times G\right) $ the equality
\begin{equation*}
\lim\limits_{n\rightarrow \infty }\left\Vert \frac{1}{n}\sum%
\limits_{l=1}^{n}\left( e^{\varphi \left( \left\vert S_{ll}\left( f\right)
-f\right\vert \right) }-1\right) \right\Vert _{C}=0
\end{equation*}%
is satisfied.

b) For any function $\varphi \in \Psi $ satisfying the condition
\begin{equation*}
\overline{\lim\limits_{u\rightarrow \infty }}\frac{\varphi \left( u\right) }{%
\sqrt{u}}=\infty
\end{equation*}%
there exists a function $F\in C\left( G\times G\right) $ such that
\begin{equation*}
\overline{\lim\limits_{m\rightarrow \infty }}\frac{1}{m}\sum%
\limits_{l=1}^{m}\left( e^{\varphi \left( \left\vert S_{ll}\left(
0,0,F\right) -F\left( 0,0\right) \right\vert \right) }-1\right) =+\infty .
\end{equation*}
\end{OldTheorem5}

For the two-dimensional Walsh-Fourier series Weisz \cite{We2} proved that if
$f\in L_{1}\left( G\times G\right) $ then
\begin{equation*}
\frac{1}{n}\sum\limits_{j=0}^{n-1}\left( S_{j,j}\left( x,y;f\right) -f\left(
x,y\right) \right) \rightarrow 0
\end{equation*}%
for a. e. $\left( x,y\right) \in G\times G$.

In the paper we consider the strong means%
\begin{equation*}
H_{n}^{p}f:=\left( \frac{1}{2^{n}}\sum\limits_{m=0}^{2^{n}-1}\left\vert
S_{mm}f\right\vert ^{p}\right) ^{1/p}
\end{equation*}%
and the maximal strong operator%
\begin{equation*}
H_{\ast }^{p}f:=\sup\limits_{n\in \mathbb{N}}H_{n}^{p}f.
\end{equation*}

We study the a. e. convergence of strong Marcinkiewicz means of the
two-dimensional Walsh-Fourier series. In particular, the following is true.

\begin{theorem}
\label{strongest}Let $f\in L\log L\left( G\times G\right) $ and $0<p\leq 2$.
Then%
\begin{equation*}
\mu \left\{ H_{\ast }^{p}f>\lambda \right\} \lesssim \frac{1}{\lambda }%
\left( 1+\iint\limits_{G\times G}\left\vert f\right\vert \log ^{+}\left\vert
f\right\vert \right) .
\end{equation*}
\end{theorem}

The weak type $\left( L\log ^{+}L,1\right) $ inequality and the usual
density argument of Marcinkiewicz and Zygmund imply

\begin{theorem}
\label{strongcon} Let $f\in L\log L\left( G\times G\right) $ and $0<p\leq 2$.
Then
\begin{equation*}
\left( \frac{1}{n}\sum\limits_{m=0}^{n-1}\left\vert S_{mm}\left(
x,y,f\right) -f\left( x,y\right) \right\vert ^{p}\right) ^{1/p}\rightarrow 0%
\text{ for a.e. }\left( x,y\right) \in G\times G\text{ as }n\rightarrow
\infty .
\end{equation*}
\end{theorem}

We note that from the theorem of Getsadze \cite{Gets} it follows that the
class $L\log L$ in the last theorem is necessary in the context of strong
summability in question. That is, it is not possible to give a larger convergence space (of the form $L\log L\phi(L)$ with $\phi(\infty)=0$) than
$L\log L$. This means a sharp contrast between the one and two dimensional strong summability.

We also note that in the case of trigonometric system Sj\"olin proved \cite{Sjo} that for every $p>1$ and two variable function $f\in L_p(\mathbb T^2)$  the almost everywhere convergence $S_{n n}f\to f$ ($n\to\infty$) holds. Since this issue with respect to the Walsh system is still open, then in this point of view
Theorem 2 may seem more interesting.

\section{Proof of Theorems}

Let $f\in L_{1}\left( G\times G\right) $. Then the dyadic maximal function
is given by
\begin{equation*}
Mf\left( x,y\right) :=\sup\limits_{n\in \mathbb{N}}2^{2n}\left\vert
\int\limits_{I_{n}\times I_{n}}f\left( x+s,y+t\right) d\mu \left( s,t\right)
\right\vert .\,\,
\end{equation*}

For a two-dimensional integrable function $f$ we need to introduce the
following hybrid maximal functions
\begin{equation*}
M_{1}f\left( x,y\right) :=\sup\limits_{n\in \mathbb{N}}2^{n}\int%
\limits_{I_{n}}\left\vert f\left( x+s,y\right) \right\vert d\mu \left(
s\right) ,
\end{equation*}%
\begin{equation*}
M_{2}f\left( x,y\right) :=\sup\limits_{n\in \mathbb{N}}2^{n}\int%
\limits_{I_{n}}\left\vert f\left( x,y+t\right) \right\vert d\mu \left(
t\right) ,
\end{equation*}%
\begin{eqnarray}
&&V_{1}\left( x,y,f\right)   \label{V1} \\
&:&=\sup\limits_{n\in \mathbb{N}}\left( \frac{1}{2^{n}}\int\limits_{G}\left(
\sum\limits_{j=0}^{n-1}2^{j-1}\mathbb{I}_{I_{j}}\left( t\right)
S^{(1)}_{2^{n}}f\left( x+t+e_{j},y\right) \right) ^{2}d\mu \left( t\right) \right)
^{1/2},  \notag
\end{eqnarray}%
\begin{eqnarray}
&&V_{2}\left( x,y,f\right)   \label{V2} \\
&:&=\sup\limits_{n\in \mathbb{N}}\left( \frac{1}{2^{n}}\int\limits_{G}\left(
\sum\limits_{j=0}^{n-1}2^{j-1}\mathbb{I}_{I_{j}}\left( t\right)
S^{(2)}_{2^{n}}f\left( x,y+t+e_{j}\right) \right) ^{2}d\mu \left( t\right) \right)
^{1/2}.  \notag
\end{eqnarray}%
It is well known that for $f\in L\log ^{+}L$ the following estimation holds%
\begin{equation}
\iint\limits_{G\times G}\left\vert Mf\left( x,y\right) \right\vert d\mu
\left( x,y\right) \lesssim 1+\iint\limits_{G\times G}\left\vert f\left(
x,y\right) \right\vert \log ^{+}\left\vert f\left( x,y\right) \right\vert
d\mu \left( x,y\right)   \label{max1}
\end{equation}%
and $\ $for $s=1,2$%
\begin{equation}
\iint\limits_{G\times G}\left\vert M_{s}f\left( x,y\right) \right\vert d\mu
\left( x,y\right) \lesssim 1+\iint\limits_{G\times G}\left\vert f\left(
x,y\right) \right\vert \log ^{+}\left\vert f\left( x,y\right) \right\vert
d\mu \left( x,y\right) .  \label{max2}
\end{equation}

Set%
\begin{equation*}
\Omega :=\left\{ \left( x,y\right) \in G\times G:V_{1}f\left( x,y\right)
>\lambda \right\} .
\end{equation*}%
Then by Fubin's Theorem and Theorem D we can write%
\begin{eqnarray}
\mu \left( \Omega \right) &=&\iint\limits_{G\times G}\mathbb I_{\Omega }\left(
x,y\right) d\mu \left( x,y\right)  \label{weakV1} \\
&=&\int\limits_{G}\left( \int\limits_{G}\mathbb I_{\Omega }\left( x,y\right) d\mu
\left( x\right) \right) d\mu \left( y\right)  \notag \\
&\lesssim &\frac{1}{\lambda }\int\limits_{G}\left( \int\limits_{G}\left\vert
f\left( x,y\right) \right\vert d\mu \left( x\right) \right) d\mu \left(
y\right)  \notag \\
&\lesssim &\frac{\left\Vert f\right\Vert _{1}}{\lambda }.  \notag
\end{eqnarray}

Analogously, we can prove that%
\begin{equation}
\mu \left\{ \left( x,y\right) \in G\times G:V_{2}f\left( x,y\right) >\lambda
\right\} \lesssim \frac{\left\Vert f\right\Vert _{1}}{\lambda }.
\label{weakV2}
\end{equation}

For Dirichlet kernel Schipp proved the following representation \cite[page
622]{Sch}%
\begin{eqnarray}
D_{m}\left( x\right) &=&\sum\limits_{k=0}^{n-1}\mathbb I_{I_{k}\backslash
I_{k+1}}\left( x\right) \sum\limits_{j=0}^{k}\varepsilon
_{kj}2^{j-1}w_{m}\left( x+e_{j}\right)  \label{sch} \\
&&-\frac{1}{2}w_{m}\left( x\right) +\left( m+1/2\right) \mathbb I_{I_{n}}\left(
x\right) ,  \notag
\end{eqnarray}%
where $m<2^{n}$ and%
\begin{equation*}
\varepsilon _{kj}=\left\{
\begin{array}{l}
-1,\ \text{if }j=0,1,...,k-1, \\
+1,\ \text{if }j=k.%
\end{array}%
\right.
\end{equation*}

\begin{proof}[Proof of Theorem \protect\ref{strongest}]
First, we prove that the following estimation holds%
\begin{eqnarray}
&&\left( \frac{1}{2^{n}}\sum\limits_{m=0}^{2^{n}-1}\left\vert S_{mm}\left(
x,y,f\right) \right\vert ^{2}\right) ^{1/2}  \label{mainest} \\
&\lesssim &V_{2}\left( x,y,M_{1}f\right) +V_{1}\left( x,y,M_{2}f\right)
+Mf\left( x,y\right)   \notag \\
&&+V_{2}\left( x,y,A\right) +V_{1}\left( x,y,A\right) +\left\Vert
f\right\Vert _{1},  \notag
\end{eqnarray}%
where $A$  is an integrable on $G\times G$ function of two variable which
will be defined below.

It is easy to show that
\begin{eqnarray*}
&&\left( \sum\limits_{m=0}^{2^{n}-1}\left\vert S_{mm}\left( x,y,f\right)
\right\vert ^{2}\right) ^{1/2} \\
&=&\left( \sum\limits_{m=0}^{2^{n}-1}\left\vert S_{mm}\left(
x,y,S_{2^{n},2^{n}}f\right) \right\vert ^{2}\right) ^{1/2} \\
&=&\left( \sum\limits_{m=0}^{2^{n}-1}\left\vert \iint\limits_{G\times
G}S_{2^{n},2^{n}}\left( x+s,y+t,f\right) D_{m}\left( s\right) D_{m}\left(
t\right) d\mu \left( s,t\right) \right\vert ^{2}\right) ^{1/2} \\
&\leq &\sup\limits_{\left\{ \alpha _{mn}\left( x,y\right) \right\}
}\left\vert \iint\limits_{G\times G}S_{2^{n},2^{n}}\left( x+s,y+t,f\right)
\sum\limits_{m=0}^{2^{n}-1}\alpha _{mn}\left( x,y\right) D_{m}\left(
s\right) D_{m}\left( t\right) d\mu \left( s,t\right) \right\vert
\end{eqnarray*}%
by taking the supremum over all $\left\{ \alpha _{mn}\left( x,y\right)
\right\} $ for which%
\begin{equation*}
\left( \sum\limits_{m=0}^{2^{n}-1}\left\vert \alpha _{mn}\left( x,y\right)
\right\vert ^{2}\right) ^{1/2}\leq 1.
\end{equation*}%
From (\ref{sch}) we can write%
\begin{equation}
\iint\limits_{G\times G}S_{2^{n},2^{n}}\left( x+s,y+t,f\right)
\sum\limits_{m=0}^{2^{n}-1}\alpha _{mn}\left( x,y\right) D_{m}\left(
s\right) D_{m}\left( t\right) d\mu \left( s,t\right)   \label{J1-J9}
\end{equation}%
\begin{equation*}
=\iint\limits_{G\times G}S_{2^{n},2^{n}}\left( x+s,y+t,f\right)
\sum\limits_{k_{1}=0}^{n-1}\sum\limits_{k_{2}=0}^{n-1}\sum%
\limits_{j_{1}=0}^{k_{1}}\sum\limits_{j_{2}=0}^{k_{2}}\mathbb I_{I_{k_{1}}\backslash
I_{k_{1}+1}}\left( s\right)
\end{equation*}%
\begin{equation*}
\times \mathbb I_{I_{k_{2}}\backslash I_{k_{2}+1}}\left( t\right) \varepsilon
_{k_{1}j_{1}}\varepsilon _{k_{2}j_{2}}2^{j_{1}+j_{2}-2}
\end{equation*}%
\begin{equation*}
\times \sum\limits_{m=0}^{2^{n}-1}\alpha _{mn}\left( x,y\right) w_{m}\left(
s+t+e_{j_{1}}+e_{j_{2}}\right) d\mu \left( s,t\right)
\end{equation*}%
\begin{equation*}
-\frac{1}{2}\iint\limits_{G\times G}S_{2^{n},2^{n}}\left( x+s,y+t,f\right)
\sum\limits_{k_{1}=0}^{n-1}\sum\limits_{j_{1}=0}^{k_{1}}\mathbb I_{I_{k_{1}}%
\backslash I_{k_{1}+1}}\left( s\right)
\end{equation*}%
\begin{equation*}
\times \varepsilon _{k_{1}j_{1}}2^{j_{1}-1}\sum\limits_{m=0}^{2^{n}-1}\alpha
_{mn}\left( x,y\right) w_{m}\left( s+t+e_{j_{1}}\right) d\mu \left(
s,t\right)
\end{equation*}%
\begin{equation*}
+\iint\limits_{G\times G}S_{2^{n},2^{n}}\left( x+s,y+t,f\right)
\sum\limits_{k_{1}=0}^{n-1}\sum\limits_{j_{1}=0}^{k_{1}}\mathbb I_{I_{k_{1}}%
\backslash I_{k_{1}+1}}\left( s\right)
\end{equation*}%
\begin{equation*}
\times \varepsilon _{k_{1}j_{1}}2^{j_{1}-1}\mathbb I_{I_{n}}\left( t\right)
\sum\limits_{m=0}^{2^{n}-1}\alpha _{mn}\left( x,y\right) w_{m}\left(
s+e_{j_{1}}\right) \left( m+1/2\right) d\mu \left( s,t\right)
\end{equation*}%
\begin{equation*}
-\frac{1}{2}\iint\limits_{G\times G}S_{2^{n},2^{n}}f\left( x+s,y+t\right)
\sum\limits_{k_{2}=0}^{n-1}\sum\limits_{j_{2}=0}^{k_{2}}\mathbb I_{I_{k_{2}}%
\backslash I_{k_{2}+1}}\left( t\right)
\end{equation*}%
\begin{equation*}
\times \varepsilon _{k_{2}j_{2}}2^{j_{2}-1}\sum\limits_{m=0}^{2^{n}-1}\alpha
_{mn}\left( x,y\right) w_{m}\left( s+t+e_{j_{2}}\right) d\mu \left(
s,t\right)
\end{equation*}%
\begin{equation*}
+\frac{1}{4}\iint\limits_{G\times G}S_{2^{n},2^{n}}\left( x+s,y+t,f\right)
\sum\limits_{m=0}^{2^{n}-1}\alpha _{mn}\left( x,y\right) w_{m}\left(
s+t\right) d\mu \left( s,t\right)
\end{equation*}%
\begin{equation*}
-\frac{1}{2}\iint\limits_{G\times G}S_{2^{n},2^{n}}\left( x+s,y+t,f\right)
\sum\limits_{m=0}^{2^{n}-1}\alpha _{mn}\left( x,y\right) w_{m}\left(
s\right)
\end{equation*}%
\begin{equation*}
\times \left( m+\frac{1}{2}\right) \mathbb I_{I_{n}}\left( t\right) d\mu \left(
s,t\right)
\end{equation*}%
\begin{equation*}
+\iint\limits_{G\times G}S_{2^{n},2^{n}}\left( x+s,y+t,f\right)
\sum\limits_{k_{2}=0}^{n-1}\sum\limits_{j_{2}=0}^{k_{2}}\mathbb I_{I_{k_{2}}%
\backslash I_{k_{2}+1}}\left( t\right)
\end{equation*}%
\begin{equation*}
\times \varepsilon _{k_{2}j_{2}}2^{j_{2}-1}\sum\limits_{m=0}^{2^{n}-1}\alpha
_{mn}\left( x,y\right) w_{m}\left( t+e_{j_{2}}\right)
\end{equation*}%
\begin{equation*}
\times \left( m+\frac{1}{2}\right) \mathbb I_{I_{n}}\left( s\right) d\mu \left(
s,t\right)
\end{equation*}%
\begin{equation*}
-\frac{1}{2}\iint\limits_{G\times G}S_{2^{n},2^{n}}\left( x+s,y+t,f\right)
\sum\limits_{m=0}^{2^{n}-1}\alpha _{mn}\left( x,y\right) w_{m}\left(
t\right)
\end{equation*}%
\begin{equation*}
\times \left( m+\frac{1}{2}\right) \mathbb I_{I_{n}}\left( s\right) d\mu \left(
s,t\right)
\end{equation*}%
\begin{equation*}
+\iint\limits_{G\times G}S_{2^{n},2^{n}}\left( x+s,y+t,f\right)
\sum\limits_{m=0}^{2^{n}-1}\alpha _{mn}\left( x,y\right)
\end{equation*}%
\begin{equation*}
\times \left( m+\frac{1}{2}\right) ^{2}\mathbb I_{I_{n}}\left( s\right)
\mathbb I_{I_{n}}\left( t\right) d\mu \left( s,t\right)
\end{equation*}%
\begin{equation*}
:=\sum\limits_{k=1}^{9}J_{k}.
\end{equation*}

It is easy to show that%
\begin{eqnarray}
\left\vert J_{9}\right\vert  &\lesssim &\left(
\sum\limits_{m=0}^{2^{n}-1}\left\vert \alpha _{mn}\left( x,y\right)
\right\vert ^{2}\right) ^{1/2}  \label{J9} \\
&&\times 2^{\left( 5/2\right) n}\iint\limits_{I_{n}\times I_{n}}\left\vert
f\left( x+s,y+t\right) \right\vert d\mu \left( s,t\right)   \notag \\
&\lesssim &2^{n/2}Mf\left( x,y\right) ,  \notag
\end{eqnarray}%
\begin{equation}
\left\vert J_{5}\right\vert \lesssim 2^{n/2}\left(
\sum\limits_{m=0}^{2^{n}-1}\left\vert \alpha _{mn}\left( x,y\right)
\right\vert ^{2}\right) ^{1/2}\left\Vert f\right\Vert _{1}\lesssim
2^{n/2}\left\Vert f\right\Vert _{1}.  \label{J5}
\end{equation}%
\begin{equation}
\left\vert J_{8}\right\vert \lesssim \iint\limits_{I_{n}\times
G}S_{2^{n},2^{n}}\left( x+s,y+t,|f|\right)   \label{J8}
\end{equation}%
\begin{equation*}
\times \left\vert \sum\limits_{m=0}^{2^{n}-1}\alpha _{mn}\left( x,y\right)
w_{m}\left( t\right) \left( m+1/2\right) \right\vert d\mu \left( s,t\right)
\end{equation*}%
\begin{equation*}
=\iint\limits_{I_{n}\times G}\left( 2^{n}\int\limits_{I_{n}}\left\vert
f\left( x+s,y+t+v\right) \right\vert d\mu \left( v\right) \right)
\end{equation*}%
\begin{equation*}
\times \left\vert \sum\limits_{m=0}^{2^{n}-1}\alpha _{mn}\left( x,y\right)
w_{m}\left( t\right) \left( m+1/2\right) \right\vert d\mu \left( s,t\right)
\end{equation*}%
\begin{equation*}
=\int\limits_{G}\left( \int\limits_{I_{n}}\left(
2^{n}\int\limits_{I_{n}}\left\vert f\left( x+s,y+t+v\right) \right\vert d\mu
\left( s\right) \right) d\mu \left( v\right) \right)
\end{equation*}%
\begin{equation*}
\times \left\vert \sum\limits_{m=0}^{2^{n}-1}\alpha _{mn}\left( x,y\right)
w_{m}\left( t\right) \left( m+1/2\right) \right\vert d\mu \left( t\right)
\end{equation*}%
\begin{equation*}
\lesssim \int\limits_{G}\left( \int\limits_{I_{n}}M_{1}f\left(
x,y+t+v\right) d\mu \left( v\right) \right)
\end{equation*}%
\begin{equation*}
\times \left\vert \sum\limits_{m=0}^{2^{n}-1}\alpha _{mn}\left( x,y\right)
w_{m}\left( t\right) \left( m+1/2\right) \right\vert d\mu \left( t\right)
\end{equation*}%
\begin{equation*}
\lesssim 2^{-n}\int\limits_{G}S_{2^{n}}^{\left( 2\right) }\left(
x,y+t,M_{1}f\right)
\end{equation*}%
\begin{equation*}
\times \left\vert \sum\limits_{m=0}^{2^{n}-1}\alpha _{mn}\left( x,y\right)
w_{m}\left( t\right) \left( m+1/2\right) \right\vert d\mu \left( t\right)
\end{equation*}%
\begin{equation*}
\lesssim 2^{-n}\left( \int\limits_{G}\left( S_{2^{n}}^{\left( 2\right)
}\left( x,y+t,M_{1}f\right) \right) ^{2}d\mu \left( t\right) \right) ^{1/2}
\end{equation*}%
\begin{equation*}
\times \left( \int\limits_{G}\left\vert \sum\limits_{m=0}^{2^{n}-1}\alpha
_{mn}\left( x,y\right) w_{m}\left( t\right) \left( m+1/2\right) \right\vert
^{2}d\mu \left( t\right) \right) ^{1/2}
\end{equation*}%
\begin{equation*}
\lesssim 2^{n/2}\left( \sum\limits_{m=0}^{2^{n}-1}\left\vert \alpha
_{mn}\left( x,y\right) \right\vert ^{2}\right) ^{1/2}V_{2}\left(
x,y,M_{1}f\right)
\end{equation*}%
\begin{equation*}
\lesssim 2^{n/2}V_{2}\left( x,y,M_{1}f\right) .
\end{equation*}

Analogously, we can prove that%
\begin{equation}
\left\vert J_{6}\right\vert \lesssim 2^{n/2}V_{1}\left( x,y,M_{2}f\right) .
\label{J6}
\end{equation}

Now, we estimate $J_{7}$. Since%
\begin{equation*}
\int\limits_{I_{n}}S_{2^{n},2^{n}}\left( x+s,y+t,\left\vert f\right\vert
\right) d\mu \left( s\right) =2^{-n}S_{2^{n},2^{n}}\left( x,y+t,\left\vert
f\right\vert \right)
\end{equation*}%
we can write

\begin{equation}
\left\vert J_{7}\right\vert \lesssim
\sum\limits_{j_{2}=0}^{n-1}\sum\limits_{k_{2}=j_{2}}^{n-1}2^{j_{2}-1}\iint%
\limits_{I_{n}\times \left( I_{k_{2}}\backslash I_{k_{2}+1}\right)
}S_{2^{n},2^{n}}\left( x+s,y+t,\left\vert f\right\vert \right)   \label{J7-1}
\end{equation}%
\begin{equation*}
\times \left\vert \sum\limits_{m=0}^{2^{n}-1}\alpha _{mn}\left( x,y\right)
w_{m}\left( t+e_{j_{2}}\right) \left( m+1/2\right) \right\vert d\mu \left(
s,t\right)
\end{equation*}%
\begin{equation*}
\lesssim \sum\limits_{j_{2}=0}^{n-1}2^{j_{2}-1}\iint\limits_{I_{n}\times
I_{j_{2}}}S_{2^{n},2^{n}}\left( x+s,y+t,\left\vert f\right\vert \right)
\end{equation*}%
\begin{equation*}
\times \left\vert \sum\limits_{m=0}^{2^{n}-1}\alpha _{mn}\left( x,y\right)
w_{m}\left( t+e_{j_{2}}\right) \left( m+1/2\right) \right\vert d\mu \left(
s,t\right)
\end{equation*}%
\begin{equation*}
\lesssim \sum\limits_{j_{2}=0}^{n-1}2^{j_{2}-1}\int\limits_{I_{j_{2}}}\left(
\int\limits_{I_{n}}S_{2^{n},2^{n}}\left( x+s,y+t,\left\vert f\right\vert
\right) d\mu \left( s\right) \right)
\end{equation*}%
\begin{equation*}
\times \left\vert \sum\limits_{m=0}^{2^{n}-1}\alpha _{mn}\left( x,y\right)
w_{m}\left( t+e_{j_{2}}\right) \left( m+1/2\right) \right\vert d\mu \left(
t\right)
\end{equation*}%
\begin{equation*}
\lesssim
2^{-n}\sum\limits_{j_{2}=0}^{n-1}2^{j_{2}-1}\int%
\limits_{I_{j_{2}}}S_{2^{n},2^{n}}\left( x,y+t,\left\vert f\right\vert
\right)
\end{equation*}%
\begin{equation*}
\times \left\vert \sum\limits_{m=0}^{2^{n}-1}\alpha _{mn}\left( x,y\right)
w_{m}\left( t+e_{j_{2}}\right) \left( m+1/2\right) \right\vert d\mu \left(
t\right)
\end{equation*}%
\begin{equation*}
\lesssim
2^{-n}\sum\limits_{j_{2}=0}^{n-1}2^{j_{2}-1}\int\limits_{I_{j_{2}}}S_{2^{n},2^{n}}%
\left( x,y+t+e_{j_{2}},\left\vert f\right\vert \right)
\end{equation*}%
\begin{equation*}
\times \left\vert \sum\limits_{m=0}^{2^{n}-1}\alpha _{mn}\left( x,y\right)
w_{m}\left( t\right) \left( m+1/2\right) \right\vert d\mu \left( t\right)
\end{equation*}%
\begin{equation*}
\lesssim
2^{-n}\int\limits_{G}\sum\limits_{j_{2}=0}^{n-1}2^{j_{2}-1}\mathbb I_{I_{j_{2}}}\left(
t\right) S_{2^{n},2^{n}}\left( x,y+t+e_{j_{2}},\left\vert f\right\vert
\right)
\end{equation*}%
\begin{equation*}
\times \left\vert \sum\limits_{m=0}^{2^{n}-1}\alpha _{mn}\left( x,y\right)
w_{m}\left( t\right) \left( m+1/2\right) \right\vert d\mu \left( t\right)
\end{equation*}%
\begin{equation*}
\lesssim \left( \int\limits_{G}\left( \sum\limits_{j_{2}=0}^{n-1}2^{j_{2}-1}%
\mathbb{I}_{I_{j_{2}}}\left( t\right) S_{2^{n},2^{n}}\left(
x,y+t+e_{j_{2}},\left\vert f\right\vert \right) \right) ^{2}d\mu \left(
t\right) \right) ^{1/2}.
\end{equation*}%
Since%
\begin{equation*}
S_{2^{n},2^{n}}\left( x,y+t+e_{j_{2}},\left\vert f\right\vert \right)
\end{equation*}%
\begin{equation*}
=2^{n}\int\limits_{I_{n}}\left( 2^{n}\int\limits_{I_{n}}\left\vert f\left(
x+u,y+t+e_{j_{2}}+v\right) \right\vert d\mu \left( u\right) \right) d\mu
\left( v\right)
\end{equation*}%
\begin{equation*}
\lesssim 2^{n}\int\limits_{I_{n}}M_{1}f\left( x,y+t+e_{j_{2}}+v\right) d\mu
\left( v\right)
\end{equation*}%
\begin{equation*}
=S_{2^{n}}^{\left( 2\right) }\left( x,y+t+e_{j_{2}},M_{1}f\right)
\end{equation*}%
from (\ref{J7-1}) we can write%
\begin{equation}
\left\vert J_{7}\right\vert   \label{J7}
\end{equation}%
\begin{equation*}
\lesssim \left( \int\limits_{G}\left( \sum\limits_{j_{2}=0}^{n-1}2^{j_{2}-1}%
\mathbb{I}_{I_{j_{2}}}\left( t\right) S_{2^{n}}^{\left( 2\right) }\left(
x,y+t+e_{j_{2}},M_{1}f\right) \right) ^{2}d\mu \left( t\right) \right) ^{1/2}
\end{equation*}%
\begin{equation*}
\lesssim 2^{n/2}V_{2}\left( x,y,M_{1}f\right) .
\end{equation*}

Analogously, we can prove that%
\begin{equation}
\left\vert J_{3}\right\vert \lesssim 2^{n/2}V_{1}\left( x,y,M_{2}f\right) .
\label{J3}
\end{equation}

For $J_{1}$ we can write%
\begin{equation}
J_{1}\lesssim
\sum\limits_{k_{1}=0}^{n-1}\sum\limits_{k_{2}=0}^{n-1}\sum%
\limits_{j_{1}=0}^{k_{1}}\sum\limits_{j_{2}=0}^{k_{2}}2^{j_{1}+j_{2}-2}
\label{J1+J2}
\end{equation}%
\begin{equation*}
\times \iint\limits_{\left( I_{k_{1}}\backslash I_{k_{1}+1}\right) \times
\left( _{I_{k_{2}}\backslash I_{k_{2}+1}}\right) }S_{2^{n},2^{n}}\left(
x+s,y+t,\left\vert f\right\vert \right)
\end{equation*}%
\begin{equation*}
\times \left\vert \sum\limits_{m=0}^{2^{n}-1}\alpha _{mn}\left( x,y\right)
w_{m}\left( s+t+e_{j_{1}}+e_{j_{2}}\right) \right\vert d\mu \left( s,t\right)
\end{equation*}%
\begin{equation*}
\lesssim
\sum\limits_{j_{1}=0}^{n-1}\sum\limits_{j_{2}=0}^{n-1}2^{j_{1}+j_{2}-2}\iint%
\limits_{I_{j_{1}}\times I_{j_{2}}}S_{2^{n},2^{n}}\left( x+s,y+t,\left\vert
f\right\vert \right)
\end{equation*}%
\begin{equation*}
\times \left\vert \sum\limits_{m=0}^{2^{n}-1}\alpha _{mn}\left( x,y\right)
w_{m}\left( s+t+e_{j_{1}}+e_{j_{2}}\right) \right\vert d\mu \left( s,t\right)
\end{equation*}%
\begin{equation*}
=\sum\limits_{j_{1}=0}^{n-1}\sum\limits_{j_{2}=0}^{j_{1}}2^{j_{1}+j_{2}-2}%
\iint\limits_{I_{j_{1}}\times I_{j_{2}}}S_{2^{n},2^{n}}\left(
x+s,y+t,\left\vert f\right\vert \right)
\end{equation*}%
\begin{equation*}
\times \left\vert \sum\limits_{m=0}^{2^{n}-1}\alpha _{mn}\left( x,y\right)
w_{m}\left( s+t+e_{j_{1}}+e_{j_{2}}\right) \right\vert d\mu \left( s,t\right)
\end{equation*}%
\begin{equation*}
+\sum\limits_{j_{1}=0}^{n-1}\sum%
\limits_{j_{2}=j_{1}+1}^{n-1}2^{j_{1}+j_{2}-2}\iint\limits_{I_{j_{1}}\times
I_{j_{2}}}S_{2^{n},2^{n}}\left( x+s,y+t,\left\vert f\right\vert \right)
\end{equation*}%
\begin{equation*}
\times \left\vert \sum\limits_{m=0}^{2^{n}-1}\alpha _{mn}\left( x,y\right)
w_{m}\left( s+t+e_{j_{1}}+e_{j_{2}}\right) \right\vert d\mu \left( s,t\right)
\end{equation*}%
\begin{equation*}
=J_{11}+J_{12}.
\end{equation*}

It is easy to show that $s+t+e_{j_{2}}\in I_{j_{2}}$ for $s\in
I_{j_{1}},t\in I_{j_{2}}$ and $j_{2}\leq j_{1}$. Hence, we can write%
\begin{equation}
J_{11}\lesssim
\sum\limits_{j_{1}=0}^{n-1}\sum\limits_{j_{2}=0}^{j_{1}}2^{j_{1}+j_{2}-2}%
\iint\limits_{I_{j_{1}}\times I_{j_{2}}}S_{2^{n},2^{n}}\left(
x+s,y+t+s+e_{j_{2}},\left\vert f\right\vert \right)   \label{J11-1}
\end{equation}%
\begin{equation*}
\times \left\vert \sum\limits_{m=0}^{2^{n}-1}\alpha _{mn}\left( x,y\right)
w_{m}\left( t+e_{j_{1}}\right) \right\vert d\mu \left( s,t\right)
\end{equation*}%
\begin{equation*}
\lesssim
2^{2n}\sum\limits_{j_{1}=0}^{n-1}\sum%
\limits_{j_{2}=0}^{j_{1}}2^{j_{1}+j_{2}-2}\iint\limits_{I_{j_{1}}\times
I_{j_{2}}}\left( \iint\limits_{I_{n}\times I_{n}}\left\vert f\left(
x+s+u,y+t+s+e_{j_{2}}+v\right) \right\vert d\mu \left( u,v\right) \right)
\end{equation*}%
\begin{equation*}
\times \left\vert \sum\limits_{m=0}^{2^{n}-1}\alpha _{mn}\left( x,y\right)
w_{m}\left( t+e_{j_{1}}\right) \right\vert d\mu \left( s,t\right)
\end{equation*}%
\begin{equation*}
\lesssim
2^{2n}\sum\limits_{j_{1}=0}^{n-1}\sum\limits_{j_{2}=0}^{j_{1}}2^{j_{2}-2}
\end{equation*}%
\begin{equation*}
\times \int\limits_{I_{j_{2}}}\left( \iint\limits_{I_{n}\times I_{n}}\left(
2^{j_{1}}\int\limits_{I_{j_{1}}}\left\vert f\left(
x+s+u,y+t+s+e_{j_{2}}+v\right) \right\vert d\mu \left( s\right) \right)
\right) d\left( u,v\right)
\end{equation*}%
\begin{equation*}
\times \left\vert \sum\limits_{m=0}^{2^{n}-1}\alpha _{mn}\left( x,y\right)
w_{m}\left( t+e_{j_{1}}\right) \right\vert d\mu \left( t\right)
\end{equation*}%
\begin{equation*}
\lesssim
2^{2n}\sum\limits_{j_{1}=0}^{n-1}\sum\limits_{j_{2}=0}^{j_{1}}2^{j_{2}-2}
\end{equation*}%
\begin{equation*}
\times \int\limits_{I_{j_{2}}}\left( \iint\limits_{I_{n}\times I_{n}}\left(
2^{j_{1}}\int\limits_{I_{j_{1}}}\left\vert f\left(
x+s,y+t+s+e_{j_{2}}+u+v\right) \right\vert d\mu \left( s\right) \right)
\right) d\left( u,v\right)
\end{equation*}%
\begin{equation*}
\times \left\vert \sum\limits_{m=0}^{2^{n}-1}\alpha _{mn}\left( x,y\right)
w_{m}\left( t+e_{j_{1}}\right) \right\vert d\mu \left( t\right)
\end{equation*}%
\begin{equation*}
\lesssim \sum\limits_{j_{1}=0}^{n-1}\sum\limits_{j_{2}=0}^{j_{1}}2^{j_{2}-2}
\end{equation*}%
\begin{equation*}
\times \int\limits_{I_{j_{2}}}\left( 2^{n}\int\limits_{I_{n}}\left(
2^{j_{1}}\int\limits_{I_{j_{1}}}\left\vert f\left(
x+s,y+t+s+e_{j_{2}}+v\right) \right\vert d\mu \left( s\right) \right)
\right) d\left( v\right)
\end{equation*}%
\begin{equation*}
\times \left\vert \sum\limits_{m=0}^{2^{n}-1}\alpha _{mn}\left( x,y\right)
w_{m}\left( t+e_{j_{1}}\right) \right\vert d\mu \left( t\right) .
\end{equation*}%
Set%
\begin{equation*}
A_{j_{1}}\left( x,y\right) :=2^{j_{1}}\int\limits_{I_{j_{1}}}\left\vert
f\left( x+s,y+s\right) \right\vert d\mu \left( s\right) .
\end{equation*}%
Then it is evident that%
\begin{equation*}
A_{j_{1}}\left( x,y+x\right) =2^{j_{1}}\int\limits_{I_{j_{1}}}\left\vert
f\left( x+s,y+x+s\right) \right\vert d\mu \left( s\right)
\end{equation*}%
\begin{equation*}
=2^{j_{1}}\int\limits_{I_{j_{1}}}\left\vert F_{2}\left( x+s,y\right)
\right\vert d\mu \left( s\right) ,
\end{equation*}%
where
\begin{equation*}
F_{2}\left( x,y\right) :=f\left( x,y+x\right) .
\end{equation*}%
From the condition of the theorem it is evident that $F_{2}\in L\log L\left(
G\times G\right) $. On the other hand,%
\begin{equation*}
\sup\limits_{j}A_{j}\left( x,x+y\right) \lesssim M_{1}F_{2}\left( x,y\right)
.
\end{equation*}%
Let%
\begin{equation*}
A\left( x,y\right) :=\sup\limits_{j}A_{j}\left( x,y\right) .
\end{equation*}

It is evident that%
\begin{equation}
\iint\limits_{G\times G}A\left( x,y\right) d\mu \left( x,y\right)
=\iint\limits_{G\times G}A\left( x,y+x\right) d\mu \left( x,y\right)
\label{LlogL}
\end{equation}%
\begin{equation*}
\lesssim \iint\limits_{G\times G}M_{1}F_{2}\left( x,y\right) d\mu \left(
x,y\right)
\end{equation*}%
\begin{equation*}
\lesssim 1+\iint\limits_{G\times G}\left\vert F_{2}\left( x,y\right)
\right\vert \log ^{+}\left\vert F_{2}\left( x,y\right) \right\vert d\mu
\left( x,y\right)
\end{equation*}%
\begin{equation*}
\lesssim 1+\iint\limits_{G\times G}\left\vert f\left( x,y\right) \right\vert
\log ^{+}\left\vert f\left( x,y\right) \right\vert d\mu \left( x,y\right) .
\end{equation*}

Then from (\ref{J11-1}) we have
\begin{equation}
\left\vert J_{11}\right\vert \lesssim
\sum\limits_{j_{1}=0}^{n-1}\sum\limits_{j_{2}=0}^{j_{1}}2^{j_{2}-2}
\label{J11}
\end{equation}%
\begin{equation*}
\times \int\limits_{I_{j_{2}}}\left( 2^{n}\int\limits_{I_{n}}A\left(
x,y+t+v+e_{j_{2}}\right) \right) d\left( v\right)
\end{equation*}%
\begin{equation*}
\times \left\vert \sum\limits_{m=0}^{2^{n}-1}\alpha _{mn}\left( x,y\right)
w_{m}\left( t+e_{j_{1}}\right) \right\vert d\mu \left( t\right)
\end{equation*}%
\begin{equation*}
\lesssim
\sum\limits_{j_{1}=0}^{n-1}\sum\limits_{j_{2}=0}^{j_{1}}2^{j_{2}-2}\int%
\limits_{I_{j_{2}}}S_{2^{n}}^{\left( 2\right) }\left(
x,y+t+e_{j_{2}},A\right)
\end{equation*}%
\begin{equation*}
\times \left\vert \sum\limits_{m=0}^{2^{n}-1}\alpha _{mn}\left( x,y\right)
w_{m}\left( t+e_{j_{1}}\right) \right\vert d\mu \left( t\right)
\end{equation*}%
\begin{equation*}
\lesssim
\sum\limits_{j_{1}=0}^{n-1}\int\limits_{G}\sum%
\limits_{j_{2}=0}^{j_{1}}2^{j_{2}-2}\mathbb{I}_{I_{j_{2}}}\left( t\right)
S_{2^{n}}^{\left( 2\right) }\left( x,y+t+e_{j_{2}},A\right)
\end{equation*}%
\begin{equation*}
\times \left\vert \sum\limits_{m=0}^{2^{n}-1}\alpha _{mn}\left( x,y\right)
w_{m}\left( t+e_{j_{1}}\right) \right\vert d\mu \left( t\right)
\end{equation*}%
\begin{equation*}
\lesssim \sum\limits_{j_{1}=0}^{n-1}\left( \int\limits_{G}\left(
\sum\limits_{j_{2}=0}^{j_{1}}2^{j_{2}-2}\mathbb{I}_{I_{j_{2}}}\left(
t\right) S_{2^{n}}^{\left( 2\right) }\left( x,y+t+e_{j_{2}},A\right) \right)
^{2}d\mu \left( t\right) \right) ^{1/2}
\end{equation*}%
\begin{equation*}
\lesssim \sum\limits_{j_{1}=0}^{n-1}2^{j_{1}/2}V_{2}\left( x,y,A\right)
\lesssim 2^{n/2}V_{2}\left( x,y,A\right) ,
\end{equation*}%
where%
\begin{equation*}
A\in L_{1}\left( G\times G\right) .
\end{equation*}%
Analogously, we can prove that%
\begin{equation}
J_{12}\lesssim 2^{n/2}V_{1}\left( x,y,A\right) .  \label{J12}
\end{equation}

Combining (\ref{J1+J2}), (\ref{J11}) and (\ref{J12}) we conclude that%
\begin{equation}
\left\vert J_{1}\right\vert \lesssim 2^{n/2}V_{1}\left( x,y,A\right)
+2^{n/2}V_{2}\left( x,y,A\right) .  \label{J1}
\end{equation}

Analogously, we can prove that%
\begin{equation}
\left\vert J_{2}\right\vert +\left\vert J_{4}\right\vert \lesssim
2^{n/2}V_{1}\left( x,y,A\right) +2^{n/2}V_{2}\left( x,y,A\right) .
\label{J2+J4}
\end{equation}

Combining (\ref{J1-J9}), (\ref{J9})-(\ref{J3}), (\ref{J1}),(\ref{J2+J4}) we
obtain of estimation (\ref{mainest}).

Since%
\begin{equation*}
H_{\ast }^{p}f\leq H_{\ast }^{2}f\text{ \ \ \ }\left( 0<p\leq 2\right) ,
\end{equation*}%
and%
\begin{equation*}
\mu \left\{ Mf>\lambda \right\} \lesssim \frac{\left\Vert f\right\Vert _{1}}{%
\lambda }
\end{equation*}

from (\ref{max1}), (\ref{max2}), (\ref{weakV1}), (\ref{weakV2}), (\ref%
{mainest}), (\ref{LlogL}) and Theorem D we conclude that%
\begin{eqnarray*}
&&\mu \left\{ H_{\ast }^{p}f>\lambda \right\}  \\
&\lesssim &\frac{1}{\lambda }\left( \left\Vert M_{1}f\right\Vert
_{1}+\left\Vert M_{2}f\right\Vert _{1}+\left\Vert A\right\Vert
_{1}+\left\Vert f\right\Vert _{1}\right)  \\
&\lesssim &\frac{1}{\lambda }\left( 1+\iint\limits_{G\times G}\left\vert
f\right\vert \log ^{+}\left\vert f\right\vert \right) .
\end{eqnarray*}

Theorem \ref{strongest} is proved.
\end{proof}

\end{document}